\documentclass[11pt]{article}
\usepackage{array, amsmath,amsthm,amsfonts,amssymb,mathrsfs,epsfig,bm}
\usepackage{subfigure}
\usepackage{lipsum}
\usepackage[usenames]{color}
\usepackage[colorlinks=true]{hyperref}
\usepackage{graphicx}
\usepackage[scale=.8]{geometry}
\usepackage{ulem}
\usepackage{textcomp}
\usepackage{amsmath}
\usepackage{esint}
\usepackage{cleveref}
\usepackage[export]{adjustbox} % For images align
\usepackage{caption}
\usepackage{authblk}
\usepackage{cite}

\newtheorem{theorem}{Theorem}[section]

\newtheorem{lemma}{Lemma}[section]

\newtheorem{corollary}{Corollary}[section]
\newtheorem{proposition}{Proposition}[section]
\newtheorem{definition}{Definition}[section]

\theoremstyle{remark}
\newtheorem{remark}{Remark}[section]

\newtheorem{example}{Example}[section]
\newtheorem{hypothesis}{Hypothesis}[section]

 \def\beqlb{\begin{eqnarray}}\def\eeqlb{\end{eqnarray}}
 \def\beqnn{\begin{eqnarray*}}\def\eeqnn{\end{eqnarray*}}

\def\N{\mathbb{N}}

\def\R{\mathbb{R}}

\renewcommand{\Phi}{\varPhi}
\renewcommand{\epsilon}{\varepsilon}

\definecolor{mygray}{gray}{0.9}
\definecolor{deeppink}{RGB}{255,20,147}
\definecolor{mygreen}{rgb}{0.05, 0.576, 0.03}
\definecolor{myred}{rgb}{0.768, 0.09, 0.09}

\newcommand{\red}{\color{red}}
\newcommand{\blue}{\color{blue}}

\long\def\symbolfootnote[#1]#2{\begingroup
\def\thefootnote{\fnsymbol{footnote}}\footnote[#1]{#2}\endgroup}
\def\comment#1{\setlength\marginparwidth{65pt} %80pt for other size
\marginpar{\raggedright\fontsize{7.81}{9}
\selectfont\upshape\hrule\smallskip
#1\par\smallskip\hrule}}

\usepackage{MnSymbol}

\usepackage{authblk}
\title{Revisiting the identification problem of a function by the ratio of Laplace transforms of powers of the function}
\begin{document}

\author{Dongdong Hu\thanks{Department of Financial and Actuarial Mathematics,
Xi'an Jiaotong-Liverpool University, P.\ R.\ China}\,\,\thanks{Yiwu Industrial \& Commercial College, P.\ R.\ China}\, , \,Linglong Yuan\thanks{ Department of Mathematical Sciences,
University of Liverpool, UK}\, , \,Minzhi Zhao\thanks{ School of Mathematical Sciences, Zhejiang University, P.\ R.\ China}}
\date{\today}
\maketitle

\abstract{ The ratio of Laplace transforms of powers of a function arises in the context of empirical auction. The question whether a function is uniquely identified by this ratio has been answered affirmatively, if the function is non-negative, non-decreasing and right analytic.
This paper extends the result to a larger class of functions without monotonicity.
%{\red This paper generalized the determination problem to the right analytic functions by considering the the concatenated exponential order and right analytic functions and applying the properties of the Laplace transform. We also derive out the necessary conditions for the determination problem in the class of right exponential analytic functions.} At last, we display several counterexamples of this identification problem deduce from the main results of this article.
A conjecture in the literature says that all c\`adl\`ag functions can be identified by the ratio. We disprove this conjecture by
 providing simple functions that cannot be identified.
}

	\vspace{8pt} \noindent {\textit{Key words:} Identification problem, Laplace transform, Convolution,  Empirical auction}
	

\section{Introduction}
Laplace transform is a fundamental tool to characterise a real-valued function. More precisely, for any function $p:[0,\infty)\mapsto \R$  of exponential  order (i.e.\ there exist $c>0,C>0$ such that $|p(x)|\leq Ce^{cx}$ for any $x\geq 0$), the Laplace transform of $p$ is defined as
\[ \widehat p(\lambda):=\int_0^{\infty} e^{-\lambda x}p(x)dx,\]
where $\lambda$ is any complex number with the real part larger than $c$.
%Through out this paper, we say two functions are identical if they differ only on a measurable set of Lebesgue measure zero.
Then $\widehat p$ uniquely identifies $p$ (up to a set of Lebesgue measure zero).
In fact, $p$ can be expressed using $\widehat p$ via the inverse Laplace transform, see \cite{widder2015laplace}.

Recently, the ratio of Laplace transforms has been the object of study in the literature. For instance, authors of \cite{YANG2017,YANG2019821} discussed the relationship between the monotonicity of the ratio of two different functions $p(x)/q(x)$ in $x$ and the monotonicity of $\widehat{p}(\lambda)/\widehat{q}(\lambda)$ in $\lambda$.

In this paper, we revisit the ratio of Laplace transforms of different powers of a function, following \cite{konstantopoulos2021does}. Let $m,n$ be any distinct  integers in $\N:=\{1,2,\cdots\}$ and $p$ be of exponential order. The ratio of Laplace transform of powers (respectively $n,m$) of a function $p$ is then defined as
\[H_{n,m}(p,\lambda):=\frac{\widehat {p^n}(\lambda)}{\widehat {p^m}(\lambda)}.\]
The ratio $H_{n,m}(p,\cdot)$ has been studied on whether it can identify the function $p$ or not, see \cite{konstantopoulos2021does, luo2023identification}.
The identification  in this context bears a weaker sense, meaning whether
\begin{equation}\label{eqn:hh}H_{n,m}(p,\cdot)=H_{n,m}(q,\cdot)\end{equation}
implies $p=q$ (i.e.\ $p(x)=q(x)$ for any $x\geq 0$) for $n,m$ given, if we know some properties of $p$ and $q$. Note that here $H_{n,m}(p,\cdot)$ is a short writing of the function $H_{n,m}(p,\lambda)$ for $\lambda$ large enough, similarly for $H_{n,m}(q,\cdot)$.  We do not consider finding the explicit expression of $p$ using $H_{n,m}(p,\cdot)$, which is much harder. The identification problem can also be reformulated via a convolution equation. Indeed, \eqref{eqn:hh} is equivalent to
\begin{equation}\label{eqn:conv}
p^n*q^m=p^m*q^n,
\end{equation}
where $*$ denotes the convolution. The problem can then be stated as this: does the above display imply $p=q$ for given $n,m$, if we know some properties of $p,q$? In \cite{konstantopoulos2021does}, the authors mainly employed \eqref{eqn:conv},
 and in this paper we will explore \eqref{eqn:hh} which could help prove stronger results than that in \cite{konstantopoulos2021does}, mainly that $p,q$ do not need to be non-decreasing functions. %$p\mapsto H_{n,m}(p,\cdot)$ is injective, for given $n,m$ and $p$ belonging to a certain set of functions. Equivalently, we ask if $H_{n,m}(p,\cdot)$ identifies $p$ knowing that $p$ belongs to a certain set. In this sense, identification means uniqueness in a set of functions.

The identification problem via $H_{n,m}(p,\cdot)$ is motivated by the studies of order statistics in empirical auction, see \cite[Section 3]{konstantopoulos2021does}. %, for more information about the order statistics, one can refer to\cite{Arnold2008,Rnyi1953},
See also \cite{luo2022order, luo2019bidding, Athey2002,guerre2019,HU2013186,Elena2011,LI2000129,LI1998139,LUO2020354,luo2023identification} for more information on the background. In this field, the convolutional form \eqref{eqn:conv} is more commonly used than \eqref{eqn:hh}, see \cite{luo2022order}, although our focus in this paper is the latter.  In the context of empirical auction, $p$ is the cumulative distribution function of a non-negative random variable (see \cite{cho2020deconvolution} for justifications of non-negativity), entailing that $p$ is non-negative and non-decreasing. For such functions with the additional condition of being right analytic on $[0,\infty)$ (which means being right analytic on every point of this interval, see \cite[Page 3]{konstantopoulos2021does} for the definition of the latter), it is proved in \cite[Theorem 1.2, 3.1]{konstantopoulos2021does} that $H_{n,m}(p,\cdot)$ identifies $p$ for any given distinct positive integers $n,m$.  The assumption of right analyticity is also used in \cite{hernandez2020estimation}. We note that although \cite[Theorem 1.2]{konstantopoulos2021does} requires  functions to be c\`adl\`ag, we believe that this assumption can be dropped.

In the context of empirical auction, $p$ is bounded. In fact, we only need $p$ to be of exponential order for it to be identifiable, see \cite[Theorem 1.2]{konstantopoulos2021does}. If $p$ is a polynomial function, then $H_{n,m}(p,\cdot)$ also identifies $p$, see \cite[Theorem 1.1]{konstantopoulos2021does}. A conjecture was proposed in \cite{konstantopoulos2021does} to enlarge the class of identifiable functions:
\begin{equation}\label{eqn:c}\text{Conjecture: all c\`adl\`ag functions could be identified by the ratio with any distinct } n, m. \end{equation}

The aim of this paper is two-fold. Firstly, we prove results which generalise \cite[Theorem 1.1]{konstantopoulos2021does}
and \cite[Theorem 1.2]{konstantopoulos2021does}. In particular, the functions do not need to be non-decreasing, as opposed to that in \cite[Theorem 1.2]{konstantopoulos2021does}. Secondly, we provide some simple counterexamples which cannot be identified by $H_{n,m}(\cdot,\cdot)$ for some distinct $n,m$, thus disproving the conjecture above.

Before showing the main results in the next section, we shall first introduce the  assumptions below that hold throughout the paper, together with justifications. % we assume for the whole paper.

\vspace{3mm}

%\subsection{Assumptions}

%{\bf Assumption 1}: \comment{we do not assume right analytic for many functions in section 3.2} {\it We consider only right analytic functions of %exponential order defined on $[0,\infty)$, unless otherwise stated.}

{\bf Assumption 1}:  {\it We shall
only consider functions $p$ that  satisfy
\begin{equation}\label{eqn:neq0}\inf\{x: p(x)\neq 0, x> 0\}=0.\end{equation}}

Indeed, as explained in \cite{konstantopoulos2021does}, if $p(x)=0$ for any $x\in[0,a]$ with some $a>0$, then
\begin{equation}\label{eqn:h=h}
 H_{n,m}(p,\cdot)=H_{n,m}(q,\cdot), \quad \forall n,m>0,
\end{equation}
where $q(x)=p(x+a), \forall x\geq 0$. To avoid this trivial case of having the same $H$ for different functions, we need to assume \eqref{eqn:neq0}.

{\bf Assumption 2}: {\it $n,m$ are co-prime in $H_{n,m}(\cdot,\cdot)$. }

If $n$ and $m$ are not co-prime, say $n=ak$ and $m=bk$, where $a,b,k$ are positive integers and $a\neq b$, then
\begin{equation*}
\begin{split}
 H_{n,m}(p,\lambda)=\frac{\widehat{p^n}(\lambda)}{\widehat{p^m}(\lambda)}=\frac{\widehat{p^{ak}}(\lambda)}{\widehat{p^{bk}}(\lambda)}=H_{a,b}(p^k,\lambda).
\end{split}
\end{equation*}
Then the best we can achieve is to identify $p^k$ with $a,b$ co-prime. Therefore we will just study the identification problem when $n,m$ are co-prime in $H_{n,m}(p,\cdot).$ %Thus identifying $p$ reduces to identifying $p^k.$ For this reason, we assume $n$ and $m$ are co-prime.

%{\bf Assumption 4}: {\it If $n,m$ are both odd numbers in $H_{n,m}(\cdot,\cdot)$, we assume the function $p$ is strictly positive near zero (not %necessarily including the point zero).}

%For any function $p$ satisfying the assumption 1 and 2,  there exist $a\neq 0$ and $k\geq 0$ such that $p(t)=at^k+o(t^{k+1})$ for $t\to0.$ In other %words, $p$ is strictly positive or negative near zero (not necessarily including zero).
%Since $n,m$ are co-prime by assumption 3, if $n$ and $m$ are odd, we have $H_{n,m}(p,\cdot)=H_{n,m}(-p,\cdot)$. To avoid this triviality of having %the same $H$ for different functions, assumption 4 needs to hold.

{\bf Assumption 3}:  {\it $n>m$ in $H_{n,m}(\cdot,\cdot)$. }

Note that
\begin{equation*}
\begin{split}
 H_{n,m}(p,\lambda)=\frac{\widehat{p^n}(\lambda)}{\widehat{p^m}(\lambda)}=\frac{1}{H_{m,n}(p,\lambda)}.
\end{split}
\end{equation*}
Then $H_{n,m}(p,\cdot)=H_{n,m}(q,\cdot)$ is equivalent to $H_{m,n}(p,\cdot)=H_{m,n}(q,\cdot)$. For simplicity, we just assume $n>m.$
%For simplicity, we assume $n>m$ in this paper.

%{\bf Assumption 6}:
%{\it We only consider functions $p$ satisfying either $p(0)=1$ or $p(0)=0.$}

%If $H_{n,m}(p,\cdot)=H_{n,m}(q,\cdot)$, then by \cite[Lemma 2.1]{konstantopoulos2021does}, we must have $p(0)^{n-m}=q(0)^{n-m}$. If $p(0)\neq 0, %q(0)\neq 0,$ we can write
%$$p(0)^{n-m}H_{n,m}(\tilde p,\cdot)=H_{n,m}(p,\cdot)=H_{n,m}(q,\cdot)=q(0)^{n-m}H_{n,m}(\tilde q,\cdot),$$
%where $\tilde p=p/p(0), \tilde q=q/q(0).$ Then $H_{n,m}(p,\cdot)=H_{n,m}(q,\cdot)$ is equivalent to $H_{n,m}(\tilde p,\cdot)=H_{n,m}(\tilde %q,\cdot)$ with $\tilde p(0)=\tilde q(0)=1$, which allows us to consider the equivalent question whether $\tilde p=\tilde q.$

All the assumptions 1-3 are assumed in the paper. We will not mention them again to avoid repetition when presenting results. There are a few more notations to introduce. The inequality  $p\not=q$ means that there exists $x\geq 0$ such that $p(x)\not=q(x)$, similarly  for $p\not=0$.
We say $p$ is positive near zero if  there is $\varepsilon>0$ such that $p(x)>0$ for all $x\in(0,\varepsilon)$. For any right analytic function $p$ satisfying assumption 1,  there exist $a\neq 0$ and $k\geq 0$ such that $p(t)=at^k+o(t^{k+1})$ for $t\downarrow 0,$ which implies that $p$ or $-p$ is  positive near zero. For any set $A$, we use $\textbf{1}_A$ to denote the indicator function of $A$, that is,
\[\textbf{1}_{A}(t)=\left\{
            \begin{array}{ll}
              1, & \hbox{$t\in A$,} \\
              0, & \hbox{$t\notin A$.}
            \end{array}
          \right.\]

\iffalse{
\comment{here we need to justification why we can do the transformation. Does the transformation affect the value of $H$?}
If $p(0)\neq0$, then $\tilde{p}(t)=\frac{p(t)}{p(0)}$ satisfies $\tilde{p}(0)=1$. Therefore we shall only consider functions that equal $1$ at zero or equal $0$ at zero, %  that are So, we can divide them into two cases. Case 1: $p(0)=1$; case 2: $p(0)=0$.

{\blue
Suppose $H_{n,m}(p,\cdot)=H_{n,m}(q,\cdot)$ and $\tilde{p}(t)=\frac{p(t)}{p(0)}$, $\tilde{q}(t)=\frac{q(t)}{q(0)}$. Then we have
\begin{equation*}
\begin{split}
H_{n,m}(\tilde{p},\lambda)&=p(0)^{m-n}\frac{\widehat{p^{n}}(\lambda)}{\widehat{p^{m}}(\lambda)}=p(0)^{m-n}H_{n,m}(p,\lambda),\\
H_{n,m}(\tilde{q},\lambda)&=p(0)^{m-n}\frac{\widehat{q^{n}}(\lambda)}{\widehat{q^{m}}(\lambda)}=p(0)^{m-n}H_{n,m}(q,\lambda),
\end{split}
\end{equation*}
therefore, $H_{n,m}(p,\cdot)=H_{n,m}(q,\cdot)$ is equivalent to $H_{n,m}(\tilde{p},\cdot)=H_{n,m}(\tilde{q},\cdot)$.

}
}\fi

%\subsection{Main Results}
%For these right analytic functions, if $p(0)\neq0$, we can do the transform $\tilde{p}(t)=\frac{p(t)}{p(0)}$ to make $\tilde{p}(0)=1$. So, we can divide them into two cases. Case 1: $p(0)=1$; case 2: $p(0)=0$. Then we can deduce the two main results of this article as follows.

%All the assumptions 1-6 are assumed in the paper. We will not mention them again to avoid repetition when presenting results.  %Next we present the main results.

\section{Main results}
%Before introducing the main results,
%\subsection{Sufficient conditions for identifiability}
In this section, we provide conditions under which functions can be identified.
{
We only consider functions $p$ satisfying either $p(0)=1$ or $p(0)=0$.
Indeed, if $H_{n,m}(p,\cdot)=H_{n,m}(q,\cdot)$ and $p(0)\not=0$, then  we can write
$$p(0)^{n-m}H_{n,m}(\tilde p,\cdot)=H_{n,m}(p,\cdot)=H_{n,m}(q,\cdot)=p(0)^{n-m}H_{n,m}(\tilde q,\cdot),$$
where $\tilde p=p/p(0), \tilde q=q/p(0).$ Then $H_{n,m}(p,\cdot)=H_{n,m}(q,\cdot)$ is equivalent to $H_{n,m}(\tilde p,\cdot)=H_{n,m}(\tilde q,\cdot)$ with $\tilde p(0)=1$, which allows us to consider the equivalent question whether $\tilde p=\tilde q.$
In this section, we only consider  right analytic functions of exponential order defined on $[0,\infty)$.} Set $x_0:=(\frac{m}{n})^{1/(n-m)}\in (0,1)$ since $n>m\geq 1$.

\begin{theorem}\label{theorem1.1}
%Suppose that $p$ and $q$ are right analytic at every point of $[0,\infty)$ and of exponential order,
%Let $p,q$ be two functions with $p(0)=q(0)=1$,
{Suppose that $p$ and $q$ are two right analytic functions of exponential order defined on $[0,\infty)$, with
$p(0)=1$,}
and  one of the following conditions holds:

\begin{enumerate}
\item[(1)] for all $t>0$, $p(t)> x_0 $ and $q(t)\ge  x_0$;

\item[(2)] for all $t>0$, $p(t)\ge 1$ and $q(t)>0$;

\item[(3)] $n$ is odd, $m$ is even,   $p(t)>1$ for all $t>0$. %{\red,    and $\inf\{t:q(t)\not =0\}=0$}.

 %then $p(t)=q(t)$ for all $t>0$.
\end{enumerate}
Then if $H_{n,m}(p,\cdot)=H_{n,m}(q,\cdot)$,
we have $p=q$.
\end{theorem}
Some counterexamples such that $p(0)=1, p\neq q$ and $H_{n,m}(p,\cdot)=H_{n,m}(q,\cdot)$ for some $n,m$:
\begin{enumerate}

\item[(a)]
$p(t)=e^{-t}\textbf{1}_{[0,T_1)}(t)+ce^{-(t-T_1)}\textbf{1}_{[T_1,T_1+T_2)}(t)+\frac{1}{2}\textbf{1}_{[T_1+T_2,\infty)}(t)$ with \[1/2<c<1 , T_1=-\ln(1-c), T_2=\ln(2c), \text{ and }
 q(t)=e^{-t}\textbf{1}_{[0,\ln 2 )}(t)+\frac{1}{2}\textbf{1}_{[\ln 2,\infty)}(t), n=2, m=1.\]

%\item[(1)] $p(t)=e^{-\frac{t}{2}}$, $q(t)=e^{-\frac{t}{2}}1_{[0,2\ln(2))}(t)+\frac{1}{2}1_{[2\ln(2),\infty)}(t), n=2, m=1$. %, then $p\neq q$ but $H_{2,1}(p,\lambda)=\frac{2\lambda+1}{2\lambda+2}=H_{2,1}(q,\lambda)$. %while $p(t)\downarrow0$ and $q(t)\downarrow\frac{1}{2}$.
%\item[(2)] $p(t)=e^{-\frac{t}{2}}$, $q(t)=e^{-\frac{t}{2}}\left[1_{[0,2)}(t)+(e-1)1_{[2,\infty)}(t)\right], n=2, m=1$.

\item[(b)]  $p(t)=1$, $q(t)=\textbf{1}_{[0,1]}(t)$, $n,m$ arbitrary.

\item[(c)] $p(t)=e^{t}$, $q(t)=e^{t}\left[\textbf{1}_{[0,1)}(t)+(e^{-1}-1)\textbf{1}_{[1,\infty)}(t)\right], n=2, m=1$. %

\item[(d)] $p(t)=e^{t}$, $q(t)=e^{t}\textbf{1}_{[0,\ln2-\frac{1}{2}\ln3)}(t)-\frac{1}{\sqrt{3}}\textbf{1}_{[\ln2-\frac{1}{2}\ln3,\infty)}(t), n=3, m=1$. % then $H_{3,1}(p,\lambda)=\frac{\lambda-1}{\lambda-3}=H_{3,1}(q,\lambda)$. %, while $p(t)\uparrow\infty$ and $q(t)\downarrow-\frac{1}{\sqrt{3}}$.

%\item[(5)] $p(t)=e^{-t}$, $q(t)=e^{-t}1_{[0,\ln(\frac{3}{2}))}(t)+\frac{2}{3}1_{[\ln(\frac{3}{2}),\infty)}(t), n=3, m=2$. %then $H_{3,2}(p,\lambda)=\frac{\lambda+2}{\lambda+3}=H_{3,2}(q,\lambda)$, while $p(t)\downarrow0$ and $q(t)\downarrow\frac{2}{3}$.

\end{enumerate}
The above counterexample are designed when we were thinking what happens if some part of a condition in Theorem \ref{theorem1.1} is not true. For instance, can we find a counterexample if there exists $t_0>0$ such that $p(t_0)\leq x_0$ with other parts in (1) unchanged (i.e. $q(t)\geq x_0$ for all $t\geq 0$)? A counterexample is given by $(a)$ where $\inf_{t\geq 0} p(t)=\lim_{t\uparrow T_1}p(t)=1-c\in(0,1/2)$ where $1/2=x_0$,  $\inf_{t\geq 0}q(t)=1/2$. Similarly if there exists $t_0>0$ such that $q(t_0)<x_0$ with other parts in (1) unchanged, a counterexample is given by {(b).}

As for condition (2), if $p(t_0)<1$ for some $t_0>0$, a counterexample is provided by {(a).} If $q(t_0)\leq 0$ for some $t_0>0$, a counterexample is provided by {(b).}

Finally for condition (3), { if $n$ is even, $m$ is odd, then a counterexample is given by (c);} if $n,m$ are both odd, then a counterexample is given by {(d).} If $p(t_0)\leq 1$ for some $t_0>0$, then a counterexample is given by {(b).}

\begin{theorem}\label{theorem1.2}
{Suppose that $p$ and $q$ are two right analytic functions of exponential order defined on $[0,\infty)$, with $p(0)=0$,}
%Let $p,q$ be two functions with $p(0)=q(0)=0$,
and  one of the following conditions holds: %Suppose that $p$ and $q$ are right analytic at every point of $[0,\infty)$ and of exponential order,
 %$p(0)=0$, and  one of the following conditions holds:

\begin{enumerate}
\item[(1)]$n$ is odd, $m$ is even,  and $p(t)\neq0$ for all $t>0$;

\item[(2)]$n$ is odd, $m$ is odd,  $p(t)\neq0$ for all $t>0$, and $\inf\{t:p(t)q(t)>0\}=0$;

\item[(3)] $n$ is even,  and  $p(t)q(t)>0$  for all $t>0$.
 %then $p(t)=q(t)$ for all $t>0$.
\end{enumerate}
Then if $H_{n,m}(p,\cdot)=H_{n,m}(q,\cdot)$,
we have $p=q$.
\end{theorem}

{\begin{remark}\label{rem:1-2}  If condition (2) holds without requiring $\inf\{t:p(t)q(t)>0\}=0$, then $H_{n,m}(p,\cdot)=H_{n,m}(q,\cdot)$ implies %If $n$ and $m$ are odd, then $H_{n,m}(p,\cdot)=H_{n,m}(-p,\cdot)$.
%Thus if   $H_{n,m}(p,\cdot)=H_{n,m}(q,\cdot)$,  and if Theorem \ref{theorem1.2} (2) holds after  deleting the condition $\inf\{t: p(t)q(t)>0\}=0$, then
$p=q$ or $p=-q$.
\end{remark}
}

Some counterexamples such that $p(0)=0, p\neq q$ and $H_{n,m}(p,\cdot)=H_{n,m}(q,\cdot)$ for some $n,m$:
\begin{enumerate}
\item[(a)] $p(t)=\sin t, q(t)=\sin t \,\textbf{1}_{[0,2\pi]}(t)$, $n, m$ arbitrary.

%\item[(1)] $p(t)=t1_{[0,2)}(t)$, $q(t)=t1_{[0,2)}(t)+(t-2)1_{[2,4)}(t), n=3,m=2$. %, then $H_{3,2}=\frac{3-3e^{2\lambda}+6\lambda+6\lambda^2+4\lambda^3}{\lambda-\lambda e^{2\lambda}+2\lambda^2+2\lambda^3}=H_{3,2}(q,\lambda)$, while $p(t)=0$ when $t\geq2$ and  $q(t)=0$ when $t\geq4$.
%\item[(2)]
 %$p(t)=(-t^2+2t)1_{[0,2)}(t)+(-t^2+6t-8)1_{[2,4)}(t),
 %q(t)=(-t^2+2t)1_{[0,2)}(t)-(-t^2+6t-8)1_{[2,4)}(t), n=3, m=1.$
%then
%\begin{equation*}
%\begin{split}
% H_{3,1}(p,\lambda)=\frac{24(15+15\lambda+6\lambda^2+\lambda^3+e^{2\lambda}(-15+15\lambda-6\lambda^2+\lambda^3))}{\lambda^4(1+e^{2\lambda}(-1+\lambda)+\lambda)}=H_{3,1}(q,\lambda),
%\end{split}
%\end{equation*}
%while $p(t)=0,q(t)=0$  when $t\geq4$.

\item[(b)]  $p(t)=t$, $q(t)=t-2\textbf{1}_{[1,\infty)}(t)$, $n=2, m=1$.  %/then $H_{2,1}(p,\lambda)=\frac{2}{\lambda}=H_{2,1}(q,\lambda)$, while $p(T)q(T)=-T^2<0$.

\end{enumerate}

Similarly as for Theorem \ref{theorem1.1}, the above counterexamples are designed when considering some part in a condition in Theorem \ref{theorem1.2} does not hold.

In condition (1), if $n$ and $m$ are both odd with other parts unchanged, then it becomes condition (2), see Remark \ref{rem:1-2}.  If $p(t_0)=0$ for some $t_0>0$, then a counterexample is from {(a).}

In condition (2), if one of $n,m$ is odd and the other is even, then it becomes condition (1) or (3).  If $p(t_0)=0$ for some $t_0>0$, then a counterexample is from {(a).}

In condition (3), if $n$ is odd, we are in the situation of condition (1) or (2). If $p(t_0)q(t_0)\le 0$ for some $t_0>0$, then a counterexample is from {(b).}

%{\red In (1), the functions $p$ and $q$ all can achieve zero at some nonzero point. They demonstrated that the condition $p(t)\not=0$ is sharp in Theorem \ref{theorem1.2} (1)-(2).
%In (1), $p(t)q(t)\ge 0$ and $p(2\pi)q(2\pi)=0$. It says that the condition $p(t)q(t)>0$ in  Theorem \ref{theorem1.2} (3) can not weakened to $p(t)q(t)\ge 0$.
%In (2), $p(t)>0$ for all $t>0$.  It shows that the condition "$n$ is odd" is necessary in  Theorem \ref{theorem1.2} (1)-(

The counterexamples given above and earlier are simple c\`adl\`ag functions that cannot be identified, and thus the conjecture \eqref{eqn:c} is disproved. % show that

Note that Theorem \ref{theorem1.1} and \ref{theorem1.2} generalise \cite[Theorem 1.1, 1.2]{konstantopoulos2021does}.  The main improvement is that Theorem \ref{theorem1.1} and \ref{theorem1.2} do not require functions to be non-decreasing, as opposed to \cite[Theorem 1.2]{konstantopoulos2021does}.   We give some functions in the remark below which are identifiable using Theorem \ref{theorem1.1} and \ref{theorem1.2} but not by \cite[Theorem 1.2]{konstantopoulos2021does} as they are not non-decreasing.

\begin{remark}\label{remark1.4}
For $k\in\N,$ let %$k\in\N$.
% Let ${\bf 1}_{[k-1,k)}(t)$ be the indicator function that
%\begin{equation*}
%{\bf 1}_{[k-1,k)}(t)=\left\{
%\begin{array}{l}
%1,\quad t\in[k-1,k),\\
%0,\quad \text{otherwise},
%\end{array} \right.
%\end{equation*}
%and
\begin{equation*}
a_k=\left\{
\begin{array}{l}
1,\quad k\text{ is odd},\\
2,\quad k\text{ is even}.
\end{array} \right.
\end{equation*}
Then
\begin{enumerate}
\item the function $p(t)=\sum_{k=1}^{\infty}a_k\textbf{1}_{[k-1,k)}(t)$ is identifiable  in the class of functions $q(t)$ characterised by (2) of Theorem \ref{theorem1.1} %{\red for which $n,m$?}. %in the class of positive right analytic functions of exponential order  satisfying assumption  2. }%uniquely determined by the ratio of Laplace transform $H_{n,m}(p,\cdot)$. However, we can not get this result from \cite[Theorem 1.2]{konstantopoulos2021does}.

\item the function $p(t)=t\textbf{1}_{[0,1)}(t)+\sum_{k=2}^{\infty}a_k\textbf{1}_{[k-1,k)}(t)$ is  identifiable in the class of functions $q(t)$ characterised by any of the three conditions in Theorem \ref{theorem1.2}. %{\red for which $n,m$ and use which condition?}.  %in the class of  positive right analytic functions of exponential order  satisfying assumptions  2.} % uniquely determined by the ratio of Laplace transform $H_{n,m}(p,\cdot)$. However, we can not get this result from \cite[Theorem 1.2]{konstantopoulos2021does}.
\end{enumerate}
\end{remark}

%\begin{remark}
%If $p,q$ are very special functions and if $H_{n,m}(p,\cdot)=H_{n,m}(q,\cdot)$, then knowing $p$ may deduce $q$ explicitly. For instance, assume $p(t)=t, t\geq 0$ and there exists $T>0$ such that
%\[q(t)=\begin{cases}p(t), \quad t\in[0,T),\\
%\sum_{i=0}^\infty b_it^i, \quad t\in(T,\infty).\end{cases}\]
%Here $|b_i|'s$ are small enough. Assume also $n=m+1$. Then if $n$ is odd, then we can show that $p=q$. If $n$ is even, then either $p=q$ or $q=t-2T1_{[T,\infty)}(t).$ Proof is omitted.
%\end{remark}

The rest of the article is organised as follows. %Section 2 demonstrates {\red the main results} and a special case of the model we discussed.
Theorem \ref{theorem1.1} and Theorem \ref{theorem1.2} are proved in Section 3. In the appendix, we discuss a simple example to investigate the relationship between $p,q$ if $p\neq q$ and \eqref{eqn:hh} holds. % would $$is discussed where the explicit expressions of the functions $p,q$ in \eqref{eqn:hh} can be found.

\section{Proofs of Theorem \ref{theorem1.1} and \ref{theorem1.2}}%Preliminaries }%Concatenated functions}
\setcounter{equation}{0}

\subsection{Technical preparations}
%The next lemma shows a special property of Laplace transform.

\begin{lemma}\label{lemma2.2} Suppose that $f:[0,\infty)\mapsto\mathbb{R}$ is of exponential order, and $N\geq 0$ and
$$f(t)=\sum_{i=0}^Na_it^i+o(t^N) \text{\, as\, }t \downarrow 0.$$% where $N$ is a nonnegative integer.
Then we have $$\widehat{f}(\lambda)=\sum_{i=0}^N\frac{a_ii!}{\lambda^{i+1}}+o\left(\frac{1}{\lambda^{N+1}}\right) \text{\, as \,}\lambda\to\infty.$$
\end{lemma}
%{\blue
%Suppose $f(t)=\sum_{i=0}^Na_it^i+o(t^N)$ as $t\downarrow0$ and of exponential order. Then let $\tilde{g}(t)=f(t)-\sum_{i=0}^Na_it^i$. Then $\tilde{g}(t)=o(t^N)$ when $t\downarrow0$, $\tilde{g}(t)$ is of exponential order when $t$ not tend to $0$. Then if we consider the formula
%\begin{equation*}
%\begin{split}
 %f(t)=\sum_{i=0}^Na_it^i+g(t)1_{[0,r)}-\sum_{i=0}^Na_it^i1_{[r,\infty)}(t)+h(t)1_{[r,\infty)}(t),
%\end{split}
%\end{equation*}
%we can let the last three term equals to $\tilde{g}(t)$, i.e.
%\begin{equation*}
%\begin{split}
% \tilde{g}(t)=g(t)1_{[0,r)}-\sum_{i=0}^Na_it^i1_{[r,\infty)}(t)+h(t)1_{[r,\infty)}(t),
%\end{split}
%\end{equation*}
%this new $\tilde{g}(t)$ is the exactly $g(t)$ in the proof below.

%}

\begin{proof}
Set $g(t)=f(t)-\sum\limits_{i=0}^Na_it^i$.  Then $g(t)=o(t^N)$ as $t\downarrow 0$, and
\begin{equation*}
\begin{split}
 \widehat{f}(\lambda)=\sum_{i=0}^N\frac{a_ii!}{\lambda^{i+1}}+\widehat{g}(\lambda).
\end{split}
\end{equation*}
We just need %So the main task left is
to show that $\widehat{g}(\lambda)=o(\frac{1}{\lambda^{N+1}})$ as $\lambda\to\infty$.

Since $g(t)=o(t^N)$ as $t\downarrow 0$,  for any $\varepsilon>0$, there exists a $\delta>0$ such that for any $t\in(0,\delta)$, we have $|g(t)|<\varepsilon t^N$. Then %Now we fix $\delta$ and get
\begin{equation}\label{lemma2.2_1}
\begin{split}
 |\widehat{g}(\lambda)|\le \int_0^{\delta}e^{-\lambda t}|g(t)|dt+\int_{\delta}^{\infty} e^{-\lambda t}|g(t)|dt=:I_1+I_2.
\end{split}
\end{equation}
For the first term, %let $s=\lambda t$,
\begin{equation*}
\begin{split}
I_1\le \varepsilon \int_0^{\delta}e^{-\lambda t}t^N\,dt
\leq \varepsilon \int_0^{\infty}e^{-\lambda t}t^N\,dt=\frac{\varepsilon N!}{\lambda^{N+1}}.
\end{split}
\end{equation*}
%Since the arbitrage of $\delta$, we have show that $\int_0^{\lambda\epsilon}e^{-\lambda t}|g(t)|dt=o(\frac{1}{\lambda^{N+1}})$ as %$\lambda\to\infty$.
To deal with the second term, we note that $g$ is of exponential order as a result of $f$ being of exponential order. Hence there are positive numbers
$C$ and $c$ such that $|g(x)|\le C e^{cx}$ for all $x\geq 0$. It implies that
 for $\lambda>c$,
\begin{equation*}
\begin{split}
 I_2\leq C\int_{\delta}^{\infty} e^{(c-\lambda)t}dt=\frac{Ce^{c \delta}}{ (\lambda-c)e^{\lambda\delta}}= { o\left(\frac{1}{\lambda^{N+1}}\right)},\, \text{\,as\,} \lambda\to\infty.
\end{split}
\end{equation*}
Since $\varepsilon$ can be arbitrarily small, combining the above two displays yields that
%Now, we have proved that $|\widehat{g}(\lambda)|=o(\frac{1}{\lambda^{N+1}})$ as $\lambda\to\infty$, which is equivalent with
$\widehat{g}(\lambda)=o(\frac{1}{\lambda^{N+1}})$ as $\lambda\to \infty$.
\end{proof}

%We start with a preliminary observation, the function
Let $f(x)=x^n-x^m$ for $x\in\R$,  and recall  $x_0=(\frac{m}{n})^{1/(n-m)}\in (0,1)$.  Note that $f'(x)=x^{m-1}(nx^{n-m}-m)$, and
\[\begin{cases}
f'(x)<0, \text{ if } x\in(0,x_0),\\
f'(x)=0, \text{ if } x=x_0,\\
f'(x)>0, \text{ if } x>x_0.
\end{cases}\]
Then we obtain the following lemma with proof omitted.

\begin{lemma}\label{lemma2.1}
 For $f$ and $x_0$ given above, %we have Let $f(x)=x^n-x^m$. Then $f(0)=f(1)=0$,
 $f(x)$ is strictly decreasing on $[0,x_0]$,  strictly increasing on $[x_0,\infty)$, and moreover %and $\lim_{x\to\infty}=\infty$.

\begin{enumerate}
  \item[(1)] if $n$ and $m$ are odd, then $f(x)$ is an odd function;
  %$f(-1)=0$, $f(x)$ increases on $(-\infty,-(\frac{m}{n})^{1/(n-m)})$, decreases on $(-(\frac{m}{n})^{1/(n-m)},0)$ and %$\lim_{x\to-\infty}f(x)=-\infty$.

\item[(2)] if $n$ is odd and $m$ is even, then $f(x)$ is strictly increasing on $(-\infty,0]$.
% and $\lim_{x\to-\infty}f(x)=-\infty$.

\item[(3)] if $n$ is even and  $m$ is odd, then $f(x)$ is strictly  decreasing on $(-\infty,0]$.
% and $\lim_{x\to-\infty}f(x)=\infty$.
\end{enumerate}

\end{lemma}

%\begin{proof}
%Clearly, $f(0)=f(1)=0$. Set  $x_0=(\frac{m}{n})^{1/(n-m)}$  which belongs to $(0,1)$.  Since  $f'(x)=x^{m-1}(nx^{n-m}-m)$,
% $f'(x_0)=0$,  $f'(x)<0$ whenever $x\in (0,x_0)$,   $f'(x)>0$ whenever $x>x_0$.  According to the relation between the derivative and the %monotonicity, we have $f(x)$ decreases on $(0,(\frac{m}{n})^{1/(n-m)})$, increases on $((\frac{m}{n})^{1/(n-m)},\infty)$. At last, since $n>m$, we %have $\lim_{x\to\infty}=\infty$.

%(1)\,If $n$ and $m$ are odds, then $f(x)$ is a odd function, the results are trivial.

%(2)\,If $n$ is odd, $m$ is even, then $n-m$ and $m-1$ are odds, which implies $f'(x)>0$ when $x<0$ and $\lim_{x\to-\infty}f(x)=-\infty$;

%(3)\,If $n$ is even, $m$ is odd, then $n-m$ is odd and $m-1$ is even, which implies $f'(x)<0$ when $x<0$ and $\lim_{x\to-\infty}f(x)=\infty$.

%\end{proof}

% We can see the graphs of the function $f(x)=x^n-x^m$ with some special choose of $n$ and $m$ in Figure \ref{Fig3-3} below.
\begin{figure}[htbp]
\centering
\subfigure[n=3,m=1]{
\includegraphics[width=4cm]{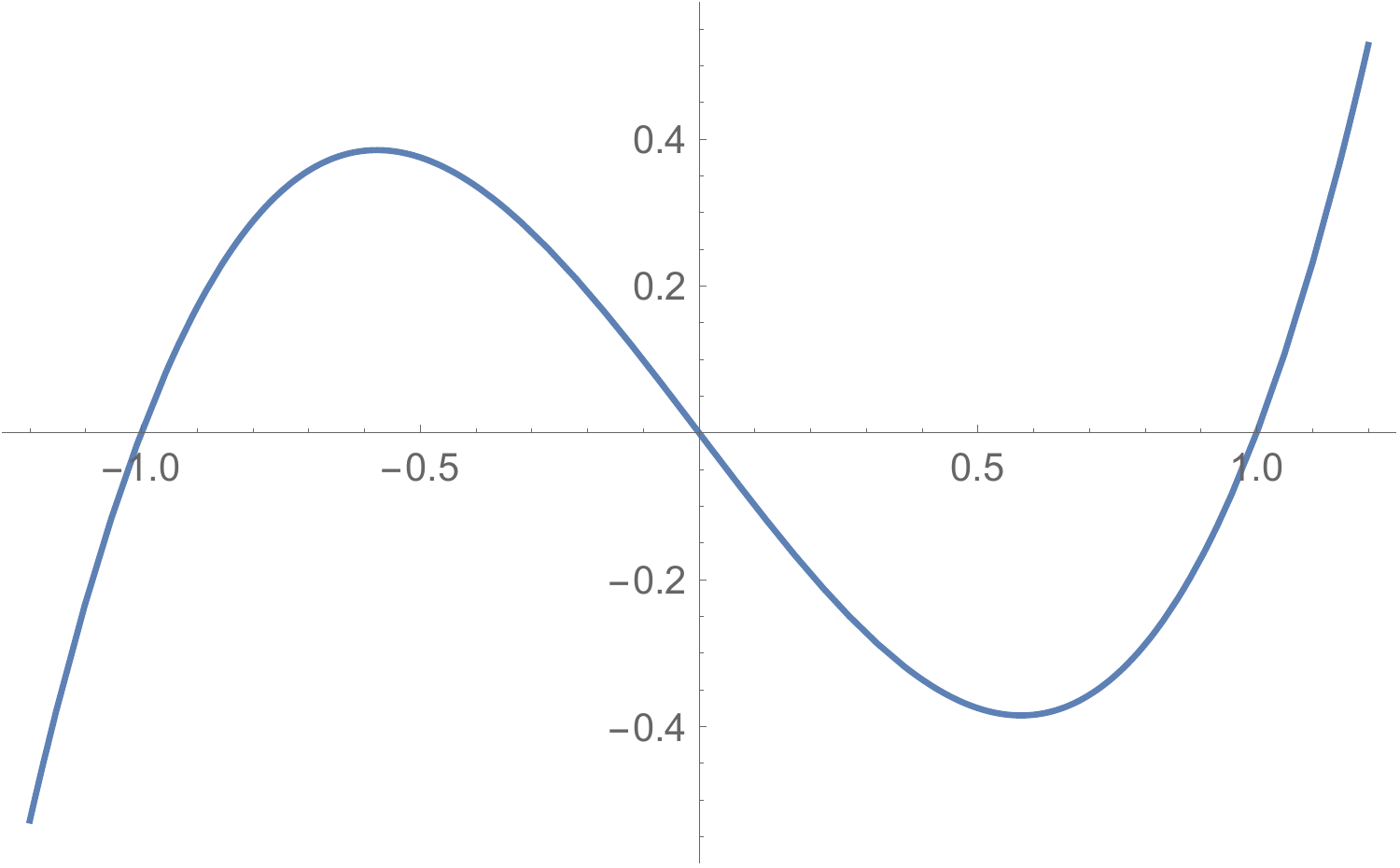}
\label{}
}
\quad
\subfigure[n=3,m=2]{
\includegraphics[width=4cm]{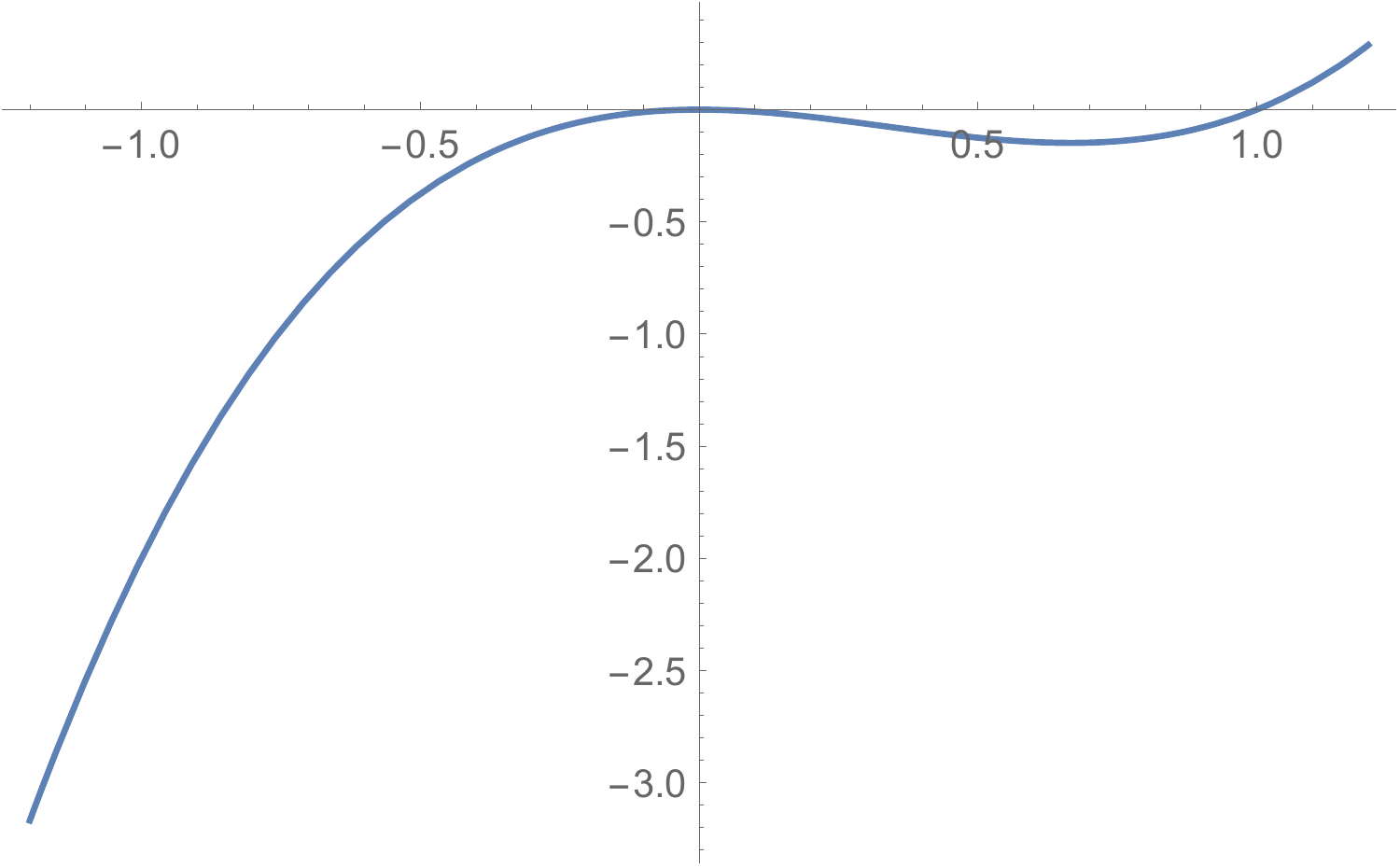}
\label{}
}
\quad
\subfigure[n=4,m=1]{
\includegraphics[width=4cm]{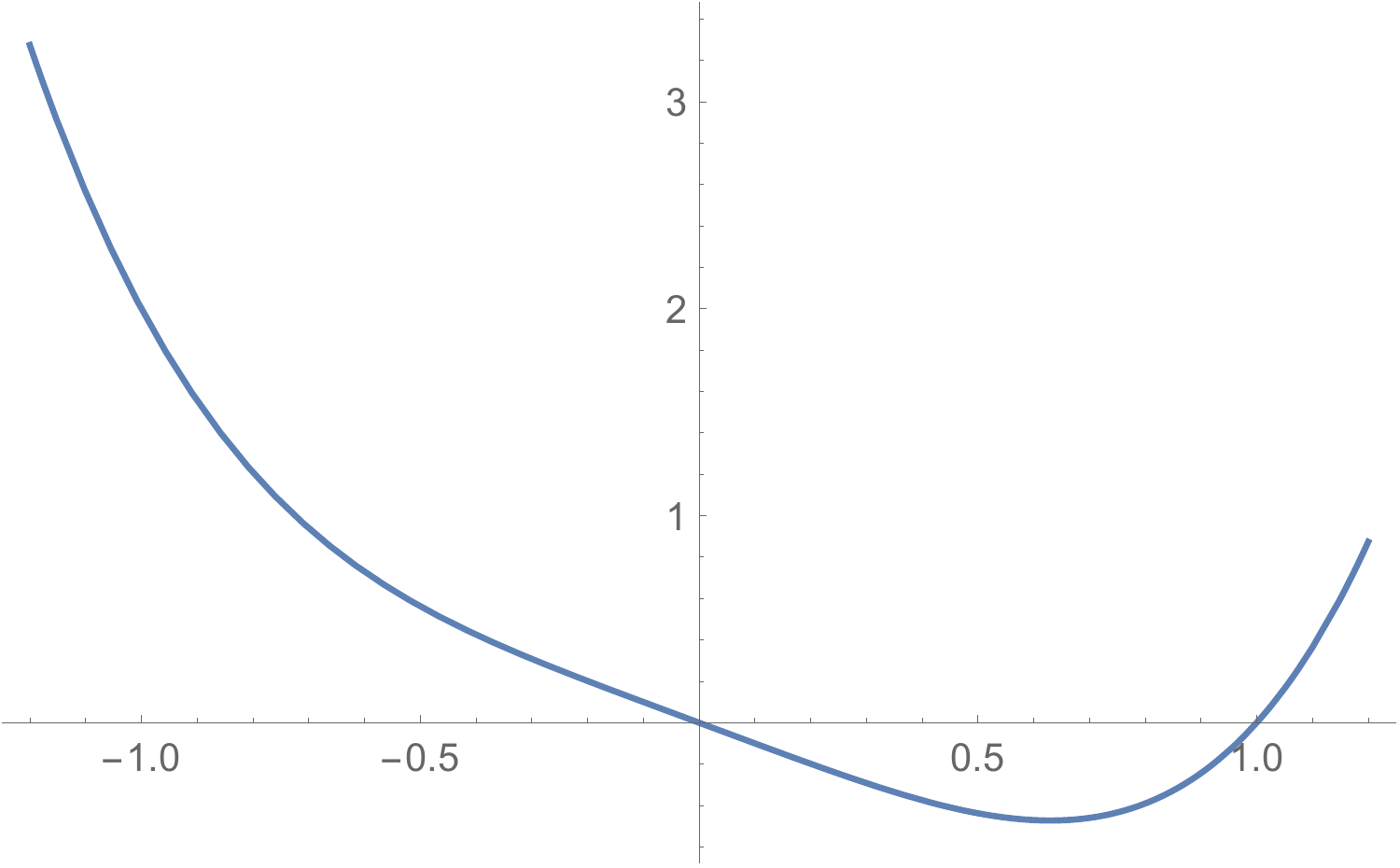}
\label{}
}

\caption{The shape of $f(x)$ for different choices of $n,m$.}\label{Fig3-3}
\end{figure}

\subsection{Concatenated functions}
%In \cite[Theorem 1.2]{konstantopoulos2021does}, the authors state and prove the uniqueness by making the additional assumption that $p$ is nonnegative nondecreasing right analytic function, it is nature to question that whether we can loose this assumption or even without it. In the proof of \cite[Theorem 1.2]{konstantopoulos2021does}, they have showed that there exists $a>0$ such that $p(t)=q(t)$ for any $t\in[0,a)$ if $H_{n,m}(p,\lambda)=H_{n,m}(q,\lambda)$, so we consider the concatenated function

In this section we introduce the main tool of this paper: concatenated functions. Let $T>0$ and assume  that  $F,G,H:[0,\infty)\mapsto\mathbb{R}$ are of exponential order % and right analytic at each point of $[0,\infty)$,
and $G\not=H$, %right analytic functions,
and $F(t)=0$ for all $t\geq T$. Consider the concatenated functions $p,q$ of the forms
\begin{equation}\label{new1.1}
\begin{split}
 p(t)&=F(t)+G(t-T)\textbf{1}_{[T,\infty)}(t),\\
 q(t)&=F(t)+H(t-T)\textbf{1}_{[T,\infty)}(t),
\end{split}
\end{equation}
%where $F(t)=f(t)1_{[0,T)}(t)$. Let $T>0$ and consider the concatenated functions $p,q$ in \eqref{new1.1}. {\red But this time , we only
We investigate what would $H_{n,m}(p,\cdot)=H_{n,m}(q,\cdot)$ imply for the relationship between $G$ and $H$. For convenience, set%The main task of this article is to find the relation between $G(t)$ and $H(t)$.

%Suppose that
%\begin{equation*}
%\begin{split}
% p(t)&=F(t)+G(t-T)1_{[T,\infty)}(t),\\
% q(t)&=F(t)+H(t-T)1_{[T,\infty)}(t),
%\end{split}
%\end{equation*}
%where $F,G,H:[0,\infty)\to\mathbb{R}$ are of exponential order, $G\not=H$, %right analytic functions,
%and $F(t)=0$ for all $t\geq T$.

%Consider the concatenated functions
%\begin{equation}\label{new2.1}
%\begin{split}
% p(t)&=F(t)+G(t-T)1_{[T,\infty)}(t),\\
 %q(t)&=F(t)+H(t-vT)1_{[T,\infty)}(t),
%\end{split}
%\end{equation}
%where $T$ is a positive number, $F,G,H:[0,\infty)\to\mathbb{R}$ are of exponential order, $F(t)=0$ for all $t\geq T$.
 %and right analytic at each point of $[0,\infty)$,
%In addition, suppose that  the sets $\{t\ge 0: F(t)\not=0\}$ and  $\{t\ge 0: G(t)\not=H(t)\}$ both have positive Lebesgue measure. %right analytic functions,
%The main task of this section  is to find the relation between $G$ and $H$
%when $H_{n,m}(p,\cdot)=H_{n,m}(q,\cdot)$.
\begin{equation}\label{newnew2.3}
 \xi(\lambda):=\widehat{F^m}(\lambda)\left(\widehat{G^n}(\lambda)-\widehat{H^n}(\lambda)\right)
 -\widehat{F^n}(\lambda)\left(\widehat{G^m}(\lambda)-\widehat{H^m}(\lambda)\right).
 \end{equation}

 \begin{lemma}\label{newlemma2.1}
Let $p,q,F,G,H$ be given in \eqref{new1.1}. The following statements are equivalent (each with $\lambda$ large enough)  \begin{align}
&H_{n,m}(p,\lambda)=H_{n,m}(q,\lambda),\\
& \frac{\widehat{G^n}(\lambda)-\widehat{H^n}(\lambda)}{\widehat{G^m}(\lambda)-\widehat{H^m}(\lambda)}=H_{n,m}(p,\lambda), \label{newnew2.2}\\
&\xi(\lambda)=e^{-\lambda T}(\widehat{H^n}(\lambda)\widehat{G^m}(\lambda)-\widehat{G^n}(\lambda)\widehat{H^m}(\lambda)). \label{newnew2.4}
\end{align}
If any of the three statements holds, we have
$\xi(\lambda)=e^{-\lambda T}o(1)$ as $\lambda\to \infty$.
\end{lemma}

\begin{proof}
For any $k\geq1$, we have
\begin{equation*}
\begin{split}
 p^k(t)&=F^k(t)+G^k(t-T)\textbf{1}_{[T,\infty)}(t),\\
 q^k(t)&=F^k(t)+H^k(t-T)\textbf{1}_{[T,\infty)}(t),
\end{split}
\end{equation*}
which implies
\begin{equation*}
\begin{split}
 \widehat{p^k}(\lambda)&=\widehat{F^k}(\lambda)+e^{-\lambda T}\widehat{G^k}(\lambda),\\
 \widehat{q^k}(\lambda)&=\widehat{F^k}(\lambda)+e^{-\lambda T}\widehat{H^k}(\lambda).
\end{split}
\end{equation*}
Thus  we have
\begin{align}
& H_{n,m}(p,\lambda)=\frac{\widehat{F^n}(\lambda)+e^{-\lambda T}\widehat{G^n}(\lambda)}{\widehat{F^m}(\lambda)+e^{-\lambda T}\widehat{G^m}(\lambda)}\label{eqn:hp},\\
& H_{n,m}(q,\lambda)=\frac{\widehat{F^n}(\lambda)+e^{-\lambda T}\widehat{H^n}(\lambda)}{\widehat{F^m}(\lambda)+e^{-\lambda T}\widehat{H^m}(\lambda)}\label{eqn:hq}.
%\end{split}
\end{align}
Then $H_{n,m}(p,\lambda)=H_{n,m}(q,\lambda)$ if and only if (\ref{newnew2.2}) holds. The equivalence between \eqref{newnew2.2} and \eqref{newnew2.4} follows from the comparison of different expressions of $H_{n,m}(p,\lambda)$ in \eqref{newnew2.2} and \eqref{eqn:hp}.
%Simple calculation shows that it  is also equivalent to
%(\ref{newnew2.4}). Consequently,
%\begin{equation}\label{new2.1}
%\begin{split}
%\widehat{F^m}(\lambda)\left(\widehat{G^n}(\lambda)-\widehat{H^n}(\lambda)\right)
 %-\widehat{F^n}(\lambda)\left(\widehat{G^m}(\lambda)-\widehat{H^m}(\lambda)\right)=e^{-\lambda %T}(\widehat{H^n}(\lambda)\widehat{G^m}(\lambda)-\widehat{G^n}(\lambda)\widehat{H^m}(\lambda)).
%\end{split}
%\end{equation}
Using \eqref{newnew2.4} and the fact that $\widehat H^k(\lambda)\xrightarrow{\lambda\to \infty}0, \widehat G^k(\lambda)\xrightarrow{\lambda\to\infty} 0$ for any $k\geq 1$, we obtain  $\xi(\lambda)=e^{-\lambda T}o(1)$ as $\lambda\to \infty$.
\end{proof}

For $p,q,F,G,H$ given in \eqref{new1.1}, % and the assumptions on $F,G,H$ given above \eqref{newnew2.3}
 we assume further that there is $N\ge 0$ such that as $t\downarrow 0$,
\begin{equation}\label{new2.8}
F(t)=\sum_{i=0}^N a_it^i+o(t^N),\,\, G(t)=\sum_{i=0}^N b_it^i+o(t^N),\,\, H(t)=\sum_{i=0}^N c_it^i+o(t^N).
\end{equation}
For $k\ge 1$ and $0\le i\le N$, denote
\begin{equation}\label{newnew2.8}
\begin{split}
 A_{k,i}&=\sum_{j_1+j_2+\cdots+j_k=i} a_{j_1}a_{j_2}\cdots a_{j_k},\\
 B_{k,i}&=\sum_{j_1+j_2+\cdots+j_k=i} b_{j_1}b_{j_2}\cdots b_{j_k},\\
 C_{k,i}&=\sum_{j_1+j_2+\cdots+j_k=i} c_{j_1}c_{j_2}\cdots c_{j_k},\\
d_i&=\sum_{j=0}^i\Big(A_{m,j}(B_{n,i-j}-C_{n,i-j})-A_{n,j}(B_{m,i-j}-C_{m,i-j})\Big)j!(i-j)!.
\end{split}
\end{equation}

\begin{proposition}\label{lemma2.3}
For $p,q$ given above, suppose that $H_{n,m}(p,\cdot)=H_{n,m}(q,\cdot)$. Then
%If \eqref{new2.8} holds,  then
 $d_i=0$ for  $i=0, 1,\cdots, N$.
%(2) If
 %$F, G, H$ are right analytic at $0$,  that is,  there is $\delta>0$ such that  for $t\in [0,\delta)$,
 % $$F(t)=\sum_{i=0}^\infty a_it^i,\,\, G(t)=\sum_{i=0}^\infty b_it^i,\,\, H(t)=\sum_{i=0}^\infty c_it^i.$$
 % Then  $d_i=0$ for  all nonnegative integer $i$.
\end{proposition}
\begin{proof} Note that for any positive integer $k$, the functions $F^k,G^k,H^k$ are of exponential order.
If \eqref{new2.8} holds, then
%Since $F(t)$ is right analytic at $0$, then $F(t)=\sum_{i=0}^\infty a_it^i=\sum_{i=0}^Na_it^i+o(t^N)$ as $t\to0$ for any positive integer $N$. In as $t\downarrow 0$ ,
$$F^k(t)=\sum_{i=0}^NA_{k,i}t^i+o(t^N), \,\,G^k(t)=\sum_{i=0}^NB_{k,i}t^i+o(t^N),\,\, H^k(t)=\sum_{i=0}^NC_{k,i}t^i+o(t^N).$$
%\begin{equation*}
%\begin{split}
% A_{k,i}=\sum_{\substack{\sum_{j=0}^Ni_j=k\\ \sum_{j=0}^Nji_j=i}}\prod_{j=0}^Na_j^{i_j}.
%\end{split}
%\end{equation*}
Applying Lemma \ref{lemma2.2}, we deduce that  as $\lambda\to \infty$,
%\begin{equation*}
%\begin{split}
 %\widehat{F^k}(\lambda)=\sum_{i=0}^NA_{k,i}\frac{i!}{\lambda^{i+1}}+o(\frac{1}{\lambda^{N+1}}),\quad\text{ as }\lambda\to\infty.
%\end{split}
%\end{equation*}
\begin{equation*}
\begin{split}
  \widehat{F^k}(\lambda)=\frac 1{\lambda}\left[\sum_{i=0}^NA_{k,i}\frac{i!}{\lambda^{i}}+o(\lambda^{-N})\right],\quad
   \widehat{G^k}(\lambda)=\frac 1{\lambda}\left[\sum_{i=0}^NB_{k,i}\frac{i!}{\lambda^{i}}+o(\lambda^{-N})\right],
   \end{split}
\end{equation*}
and
 $$\widehat{H^k}(\lambda)=\frac 1{\lambda}\left[\sum_{i=0}^NC_{k,i}\frac{i!}{\lambda^{i}}+o(\lambda^{-N})\right].$$
%as $\lambda\to\infty$ for any positive integer $k$ and $N$, where
%\begin{equation*}
%\begin{split}
% B_{k,i}=\sum_{\substack{\sum_{j=0}^Ni_j=k\\ \sum_{j=0}^Nji_j=i}}\prod_{j=0}^Nb_j^{i_j},\quad C_{k,i}=\sum_{\substack{\sum_{j=0}^Ni_j=k\\ %\sum_{j=0}^Nji_j=i}}\prod_{j=0}^Nc_j^{i_j}.
%\end{split}
%\end{equation*}
Plugging these   expressions %$\widehat{F^k}(\lambda),\widehat{G^k}(\lambda),\widehat{H^k}(\lambda)$
into formula \eqref{newnew2.3} %, which
yields
%\begin{equation}\label{2.2}
%\begin{split}
%&\widehat{G^n}(\lambda)\widehat{F^m}(\lambda)+\widehat{F^n}(\lambda)\widehat{H^m}(\lambda)-\widehat{H^n}(\lambda)\widehat{F^m}(\lambda)-\widehat{F^n}(\lambda)\widehat{G^m}(\lambda)\\
% %=&(\sum_{i=0}^NB_{n,i}\frac{i!}{\lambda^{i+1}})(\sum_{i=0}^NA_{m,i}\frac{i!}{\lambda^{i+1}})+(\sum_{i=0}^NA_{n,i}\frac{i!}{\lambda^{i+1}})(\sum_{i=0}^NC_{m,i}\frac{i!}{\lambda^{i+1}})\\
% %-&(\sum_{i=0}^NC_{n,i}\frac{i!}{\lambda^{i+1}})(\sum_{i=0}^NA_{m,i}\frac{i!}{\lambda^{i+1}})-(\sum_{i=0}^NA_{n,i}\frac{i!}{\lambda^{i+1}})(\sum_{i=0}^NB_{m,i}\frac{i!}{\lambda^{i+1}})+o(\frac{1}{\lambda^{N+1}}),
%\end{split}
%\end{equation}
%as $\lambda\to\infty$. We can rewrite \eqref{2.2} as follows after some calculations:
\begin{equation*}
\begin{split}
 \xi(\lambda)
 =\frac {1}{\lambda^2}\left[\sum_{i=0}^{N}\frac{d_i}{\lambda^{i}}+o(\lambda^{-N})\right],\quad\text{as }\lambda\to \infty.
\end{split}
\end{equation*}
On the other hand, by Lemma \ref{newlemma2.1},
%$H_{n,m}(p,\cdot)=H_{n,m}(q,\cdot)$ implies that
 $\xi(\lambda)=e^{-\lambda T}o(1)$ as $\lambda\to \infty$.
Together with the above display we obtain % It follows that as $\lambda\to \infty$,
 %$\sum_{i=0}^{N}\frac{d_i}{\lambda^{i}}=\lambda^2 e^{-\lambda T}o(1)+o(\frac{1}{\lambda^{N}})$ as $\lambda\to\infty$.
%$Both sides times  $t^N$ gives
  $$\sum_{i=0}^{N}d_i\lambda^{N-i}=\lambda^{N+2} e^{-\lambda T}o(1)+o(1)=o(1),\quad \lambda\to\infty.$$
  Consequently, $d_i=0$ for $i=0, 1,\cdots, N$.
%(2) follows from (1) directly.
\end{proof}

Recall $x_0=(\frac{m}{n})^{1/(n-m)}$,
$T>0$, $F,G,H:[0,\infty)\to\mathbb{R}$ are of exponential order, and $F(t)=0$ for all $t\geq T$.
 %and right analytic at each point of $[0,\infty)$,
%In this section,
Assume further that  $F, G, H$ are right analytic at $0$. That is,  there is $\delta>0$ such that  for $t\in [0,\delta)$,
 \begin{equation}\label{eqn:fgh}F(t)=\sum_{i=0}^\infty a_it^i,\,\, G(t)=\sum_{i=0}^\infty b_it^i,\,\, H(t)=\sum_{i=0}^\infty c_it^i.\end{equation}
Assume also that $\inf\{t\ge 0: G(t)\not=H(t)\}=0$. Together with Assumption 1 for {$p$}, we have $l,w\ge 0$ such that
\begin{equation}\label{eqn:lw}a_w\not=0, a_i=0, \forall i<w; \quad b_l\not=c_l,  b_i=c_i, \forall i<l.\end{equation}
%clearly the sets $\{t\ge 0: F(t)\not=0\}$ and $\{t\ge 0: G(t)\not=H(t)\}$ both have positive Lebesgue measure.
For $k\ge 1$ and $i\ge 0$, recall $A_{k,i}, B_{k,i},C_{k,i}$ and $d_i$ defined in (\ref{newnew2.8}). The next lemma studies the relationship between $G(0)$ and $H(0)$ when $F(0)=1.$

%By this Lemma, we can  immediately get the following Proposition.
%\begin{proposition} \label{propnew2.1}
%If
 %$F, G, H$ are right analytic at $0$,  that is,  there is $\delta>0$ such that  for $t\in [0,\delta)$,
  %$$F(t)=\sum_{i=0}^\infty a_it^i,\,\, G(t)=\sum_{i=0}^\infty b_it^i,\,\, H(t)=\sum_{i=0}^\infty c_it^i.$$
  %Then  $d_i=0$ for  all nonnegative integer $i$, where the expression of $d_i$ is as (\ref{newnew2.8}).
%\end{proposition}

\begin{lemma}\label{lemma2.4}
%Suppose that $F, G$ and $H$ are right analytic at $0$.
Under the assumptions given above, if $H_{n,m}(p,\cdot)=H_{n,m}(q,\cdot)$ and  $F(0)=1$, then we have the following properties.
\begin{itemize}
  \item [(1)]If $G(0)=H(0)=u$, then $u=0$ or  $u^{n-m}=\frac{m}{n}$.
  \item [(2)] If $G(0)=u,H(0)=v$ and $u> v$, then \begin{equation}\label{eqn:uv}u^n-u^m=v^n-v^m,\end{equation} and moreover
\begin{enumerate}
\item[(i)]  $v<x_0$;

\item[(ii)] if $u\ge 1$, then  $v\le 0$;

\item[(iii)] if $n$ is odd and $m$ is even, then $u\leq1$.
\end{enumerate}

\end{itemize}

\end{lemma}

\begin{proof}

By Proposition \ref{lemma2.3}, we have $d_i=0$ for any integer $i\geq 0$.

(1). If $G(0)=H(0)=u$, then by \eqref{eqn:fgh}, we have $b_0=c_0=u$. %Suppose that $b_i=c_i$ for all $i<l$, and $b_l\neq c_l$.
Recall $l$ in \eqref{eqn:lw}. Then $l\ge 1$, and for any $k\geq 1$, we have $B_{k,i}=C_{k,i}$ for all $i<l$ and $B_{k,l}-C_{k,l}=ku^{k-1}(b_l-c_l)$. The fact $F(0)=1$ yields that $a_0=1$ and $A_{k,0}=1$ for all $k\ge 1$.
Hence
\begin{equation*}
\begin{split}
 0=d_l&=\sum_{j=0}^l\Big(A_{m,j}(B_{n,l-j}-C_{n,l-j})-A_{n,j}(B_{m,l-j}-C_{m,l-j})\Big)j!(l-j)!\\
 &=\Big(A_{m,0}(B_{n,l}-C_{n,l})-A_{n,0}(B_{m,l}-C_{m,l})\Big)l!\\
 &=\Big(nu^{n-1}(b_l-c_l)-mu^{m-1}(b_l-c_l)\Big)l!.
\end{split}
\end{equation*}
Thus  $u=0$ or $u^{n-m}=\frac{m}{n}$.

(2). If $G(0)=u$, $H(0)=v$ and $u>v$, then $b_0=u, c_0=v$ which entails  $B_{k,0}-C_{k,0}=u^k-v^k$ for all $k\ge 1$.
Thus
\begin{equation*}
\begin{split}
 0=d_0=A_{m,0}(B_{n,0}-C_{n,0})-A_{n,0}(B_{m,0}-C_{m,0})=(u^n-v^n)-(u^m-v^m),
\end{split}
\end{equation*}
which implies $u^n-v^n=u^m-v^m$, or $u^n-u^m=v^n-v^m$.

(i) If $u>v\ge x_0$, then according to Lemma $\ref{lemma2.1}$,
 $u^n-u^m>v^n-v^m$, which contradicts \eqref{eqn:uv}. Thus $v<x_0.$

 (ii) If $u\ge 1$ and $v>0$, then by (i), we have  {$v< x_0$}.
Using Lemma $\ref{lemma2.1}$ again, $u^n-u^m\ge 0>v^n-v^m$,  contradicting \eqref{eqn:uv}. Thus $v\leq 0$.

(iii) If  $n$ is odd, $m$ is even and $u>1$ with $u>v$, then by (ii), we must have $v\leq0$.
Applying Lemma $\ref{lemma2.1}$ yields $u^n-u^m>0\ge v^n-v^m$ which is again a contradiction. Thus $u\le 1$.
\end{proof}

The next lemma studies the same problem when $F(0)=0.$

\begin{lemma}\label{lemma2.5}
Under the same assumptions as in Lemma \ref{lemma2.4}, if $H_{n,m}(p,\cdot)=H_{n,m}(q,\cdot)$ and $F(0)=0$, {then we have the following properties:

(1) If $G(0)=H(0)=u$, then $u=0$.

(2) If $G(0)\neq H(0)$, then $n$ must be even and $G(0)=-H(0)$.

}
%  $G(0)=H(0)=0$ when $n$ is odd {\red $n$ can also be even?}, and $G(0)=-H(0)$ when $n$ is even.
\end{lemma}

\begin{proof}
Similarly, Proposition \ref{lemma2.3} entails that $d_i=0$ for any integer $i\geq 0$.
% Suppose that $a_w\not=0$ and $a_i=0$ for any $i<w$.
Recall $l,w$ from \eqref{eqn:lw}. The fact $F(0)=0$ leads to $w\ge 1$.
Also we have $A_{k,i}=0$ for $i<kw$ and $A_{k,kw}=a^k_w\not=0$, for any $k\geq 1$.
 %first nonzero term of $F(t)$ is $a_ut^u$, which means
%\begin{equation*}
%\begin{split}
 %F(t)=\sum_{i=u}^{\infty}a_it^i,
%\end{split}
%\end{equation*}
%for any $k$, let
%\begin{equation*}
%\begin{split}
 %F^k(t)=\sum_{i=ku}^{\infty}A_{k,i}t^i,
%\end{split}
%\end{equation*}
%we can assume $a_u=1$ and $A_{k,ku}=1$.

%Without loss of generality, we just prove the positive case, otherwise we can define $\tilde{F}(t)=-F(t),\tilde{G}(t)=-G(t),\tilde{H}(t)=-H(t)$ to %prove the negative case
Note that $b_i=c_i$ for $i<l$ and $b_l\neq c_l$.
If {$G(0)=H(0)=u$}, then $l\ge 1$,  and for any $k\geq 1$, we have $B_{k,i}=C_{k,i}$ for $i<l$ and { $B_{k,l}-C_{k,l}=ku^{k-1}(b_l-c_l)$}.  Thus
\begin{equation}\label{eqn:i-i}
\begin{split}
 0=d_{mw+l}=:I_1+I_2+I_3,
\end{split}
\end{equation}
where
\begin{equation*}
\begin{split}
 I_1&=\sum_{j=0}^{mw-1}\Big(A_{m,j}(B_{n,mw+l-j}-C_{n,mw+l-j})-A_{n,j}(B_{m,mw+l-j}-C_{m,mw+l-j})\Big)j!(mw+l-j)!,\\
 I_2&=\Big(A_{m,mw}(B_{n,l}-C_{n,l})-A_{n,mw}(B_{m,l}-C_{m,l})\Big)(mw)!l!
\end{split}
\end{equation*}
and
\begin{equation*}
\begin{split}
 I_3&=\sum_{j=mw+1}^{mw+l}\Big(A_{m,j}(B_{n,mw+l-j}-C_{n,mw+l-j})-A_{n,j}(B_{m,mw+l-j}-C_{m,mw+l-j})\Big)j!(mw+l-j)!.
\end{split}
\end{equation*}
To analysis $I_1,I_2,I_3$, we will need the following two facts:
\begin{equation}\label{eqn:aki}
A_{k,i}=0,\quad  \forall k\geq 1, i< kw,
\end{equation}
and
\begin{equation}\label{eqn:bc}
B_{k,i}=C_{k,i},\quad  \forall k\geq 1, i<l.
\end{equation}
Using (\ref{eqn:aki}), we obtain $I_1=0$. Using (\ref{eqn:bc}), we obtain $I_3=0$. Using the fact $mw<nw$ and (\ref{eqn:aki}), we have
\begin{equation*}
\begin{split}
 I_2=A_{m,mw}(B_{n,l}-C_{n,l})(mw)!l!=a_w^mnu^{n-1}(b_l-c_l)(mw)!l!.
\end{split}
\end{equation*}
The above three displays and (\ref{eqn:i-i}) entail
\begin{equation*}
\begin{split}
 a_w^mnu^{n-1}(b_l-c_l)(mw)!l!=0,
\end{split}
\end{equation*}
which implies $u=0$.

Now consider the case {$G(0)\neq H(0)$. Setting $u=G(0)$ and $v=H(0)$}, we have $B_{k,0}-C_{k,0}=u^k-v^k$. Then {by the similar calculations as above}
\begin{equation*}
\begin{split}
0=d_{mw}&=\sum_{j=0}^{mw}\Big(A_{m,j}(B_{n,mw-j}-C_{n,mw-j})-A_{n,j}(B_{m,mw-j}-C_{m,mw-j})\Big)j!(mw-j)!\\
 &=A_{m,mw}\Big(B_{n,0}-C_{n,0}\Big)(mw)!\\
 &=a^m_w(u^n-v^n)(mw)!,
\end{split}
\end{equation*}
and hence $u^n=v^n$. {If $n$ is odd, it follows that $u=v$, which contradicts the assumption $u\neq v$.} Consequently,
 $n$ must be even  and {$v=-u$}. The proof for this lemma is finished.
\end{proof}

\subsection{Proofs of Theorem \ref{theorem1.1} and \ref{theorem1.2}}
\begin{proof}[Proof of Theorem \ref{theorem1.1}]
 By Lemma 2.1 and Lemma 2.2 of \cite{konstantopoulos2021does}, there exists $\epsilon>0$ such that $p(t)=q(t)$ for $t\in[0,\epsilon)$. Next, let $T=\inf\{t\geq0:p(t)\neq q(t)\}$. It suffices to show that $T=\infty$.  If $T<\infty$, we let
\begin{equation*}
\begin{split}
 F(t)=p(t)\textbf{1}_{[0,T)}(t),\quad G(t)=p(t+T),\quad H(t)=q(t+T), \quad t\geq 0.
\end{split}
\end{equation*}
Then $F,G,H$ are all right analytic at $0$,  $F(t)=0$ for $t\geq T$ and $\inf\{t\ge 0: G(t)\not=H(t)\}=0$. Moreover, we have $p(t)=F(t)+G(t-T)\textbf{1}_{[T,\infty)}(t)$, $q(t)=F(t)+H(t-T)\textbf{1}_{[T,\infty)}(t)$, and
 $F(0)=1$, $G(0)=p(T)$, $H(0)=q(T)$.

 If $p(T)=q(T)$, then (1) of Lemma \ref{lemma2.4} contradicts any of the three conditions in {Theorem } \ref{theorem1.1}.  If $p(T)\neq q(T)$ (without loss of generality, we assume $p(T)> q(T)$), then (2)-(i) (resp.\ (2)-(ii), (2)-(iii)) of Lemma \ref{lemma2.4} contradicts (1) (resp.\ (2), (3)) of Theorem \ref{theorem1.1}. These contradictions entail that $T=\infty.$ Then the proof is finished.
% Next we prove all three cases $(1), (2),(3)$. Recall $x_0=(\frac{m}{n})^{1/(n-m)}$.

% (1). If $p(T)=q(T)$, then we must have $p(T)=q(T)=0$ or $p(T)^{n-m}=q(T)^{n-m}=\frac{m}{n}$, using (1) of Lemma \ref{lemma2.4}. But this contradicts the assumption in Theorem \ref[(1)]\ref{theorem1.1}. Therefore $p(T)\neq q(T).$

% $G(0)>x_0$ and $H(0)\ge x_0$.
%In the case (2),  $G(0)\ge 1$ and $H(0)>0$.
%In the case (3),   $n$ is odd, $m$ is even and $G(0)>1$.
%They all contradict  to Lemma \ref{lemma2.4}. The contradiction shows that $T=\infty$ as desired.
% $t>0$, which means $G(0)$ and $H(0)$ both bigger than $(\frac{m}{n})^{1/(n-m)}$, it is contradict with i) of Lemma \ref{lemma2.4}.
%ii)if $F(0)=1$, $n$ is odd, $m$ is even and $p(t)>1$ for all $t>0$, which means $G(0)>1$, it is contradict with ii) of Lemma \ref{lemma2.4}.
%iii)if $F(0)=1$, $p(t)>1$ and $q(t)\geq0$ for all $t>0$, which means $G(0)>1$ and $H(0)\geq0$, it is contradict with iii) of Lemma \ref{lemma2.4}.
\end{proof}

\begin{proof}[Proof of Theorem \ref{theorem1.2}]
The proof is almost identical to that of Theorem \ref{theorem1.1}. The only main difference is that we shall use Lemma \ref{lemma2.5} rather than Lemma \ref{lemma2.4}. Details are omitted.
\end{proof}
%\begin{proof}[Proof of Theorem \ref{theorem1.2}]
%Similar as the proof of Theorem \ref{theorem1.1}, there exists $\epsilon>0$ such that $p(t)=q(t)$ for $t\in[0,\epsilon)$ and let $p(t)=F(t)+G(t-T)$, %$q(t)=F(t)+H(t-T)$ with the same notation with the proof of Theorem \ref{theorem1.1}.
%i)if $n$ is odd, $p(t)\neq0$ for $t>0$, which means $G(0)\neq0$, it is contradict with i) of Lemma \ref{lemma2.5}.

%ii)if $n$ is even, $p(t)q(t)>0$ for $t>0$, which means $G(0)H(0)>0$, it is contradict with ii) of Lemma \ref{lemma2.5}.

%\end{proof}

%\section{Counterexamples}

\section{Acknowledgement}
M.Z is supported by
the NSFC grant with number 12271351. The authors thank Yao Luo for reading a draft version and helpful comments.

\iffalse{\section{Conclusions}
In this article, we proposed two criteria for the identification problem of $p$ by the ratio of Laplace transform $H_{n,m}(p,\cdot)$ by two distinct positive integers $n,m$ in right analytic functions, which have generalized the results of \cite{konstantopoulos2021does}. {\red To obtain this, we first consider the concatenated exponential order and right analytic functions, which the first part is identical and the second part is different, and then apply the properties of the Laplace transform to extend the identical part to infinity. The second contribution of this article is that we deduced the necessary condition of the determination problem in the right exponential analytic functions case, which is a generalization of piecewise polynomial functions.} What's more, we also derived out several counterexamples by violating the conditions in the criterions.

There are still many open questions about the identification problem remain to solve. In the future, we will continue studying the property of the ratio of Laplace transform of powers of functions, trying to generalized the identification problem to more general class of functions.}\fi

%Secondly, we will try to combine our new results with the financial fields, such as in the auction theory, to apply this results to solve more particle problems.

\section{Appendix: a special case}
We would like to get an idea of the relationship between $p,q$ if $p\neq q$ and $H_{n,m}(p,\cdot)=H_{n,m}(q,\cdot).$ For this purpose we study a special and simple case.

%If $p,q$ are very special functions, we can find their explicit forms given (\ref{eqn:hh})
Let $\mathcal D$ be the set of functions $f:[0,\infty)\mapsto\R$ such that
$f$ can be written as
$ f(t)=\sum_{i=0}^\infty a_it^i$,  and   $|a_k|\le \frac{K\alpha^k}{k!}$
for some $\alpha>0$ and  $K>0$ and any integer $k\geq 0$.
For any $f$ in $\mathcal D$, it is of exponential order and for $\lambda>\alpha$, we have (See \cite{Schiff1999})
\begin{equation*}
\begin{split}
 \widehat{f}(\lambda)=\sum_{k=0}^{\infty}a_{k}\widehat{t^k}=\sum_{k=0}^{\infty}\frac{a_kk!}{\lambda^{k+1}}.
\end{split}
\end{equation*}
 %This means the Laplace transform of $f\in \mathcal D$ can be done term by term.
If $f,g\in \mathcal D$, then the functions
 $f(t)g(t)$, $f'$ and $f_T(t)$ are in $\mathcal D$, where $f_T(t):=f(t+T), t\geq 0$.

%Consider the functions in $\mathcal D$
%$T>0$, $F(t)=f(t)1_{\{t<T\}}$,
%$$f(t)=\sum_{i=0}^\infty a_it^i,\,\, G(t)=\sum_{i=0}^\infty b_it^i,\,\, H(t)=\sum_{i=0}^\infty c_it^i,\quad t\geq 0,$$
%with %$\inf\{t\ge 0:f(t)\not=0\}=0$ and
%{\red $G\not=H$}.
% Let
%$$f_T(t)=f(t+T),\,\,t\ge 0.$$
For fixed $T>0$, consider concatenated functions $p,q$  in (\ref{new1.1}).
%\begin{equation}\label{new1.1}
%\begin{split}
 %p(t)&=F(t)+G(t-T)1_{[T,\infty)}(t),\\
 %q(t)&=F(t)+H(t-T)1_{[T,\infty)}(t),
%\end{split}
%\end{equation}
We assume further that
$$F(t)=f(t)\textbf{1}_{[0,T)}(t), \,\,f(t)=\sum_{i=0}^\infty a_it^i,\,\, G(t)=\sum_{i=0}^\infty b_it^i,\,\, H(t)=\sum_{i=0}^\infty c_it^i,$$
with $f,G,H\in \mathcal D$, $f\not=0$ and $G\not=H$. %$\inf\{t\ge 0:f(t)\not=0\}=0$ and
%{\red $G\not=H$}.
%Introducing such $p,q$ would allow to find equivalent conditions for $H_{n,m}(p,\cdot)=H_{n,m}(q,\cdot).$
Then equation (\ref{newnew2.4}) becomes
\begin{eqnarray*}%\label{newnew5.1}
\begin{split}
&\,\widehat{f^{m}}(\lambda)
\left(
\widehat{G^n}(\lambda)-\widehat{H^n}(\lambda)\right)
-\widehat{f^{n}}(\lambda)\left(\widehat{G^m}(\lambda)-\widehat{H^m}(\lambda)\right) \\
=&e^{-\lambda T} \Big[\widehat{H^n}(\lambda)\widehat{G^m}(\lambda)-\widehat{G^n}(\lambda)\widehat{H^m}(\lambda)+\widehat{f^{m}_T}(\lambda)
\left(
\widehat{G^n}(\lambda)-\widehat{H^n}(\lambda)\right)
-\widehat{f^{n}_T}(\lambda)\left(\widehat{G^m}(\lambda)-\widehat{H^m}(\lambda)\right)\Big].
\end{split}
\end{eqnarray*}
This, together with Lemma  \ref{newlemma2.1}, we obtain the following Proposition with proof omitted.

\begin{proposition}\label{prop4.2}
 (1)If $ H_{n,m}(p,\cdot)=H_{n,m}(q,\cdot)$, then $ H_{n,m}(p,\cdot)=H_{n,m}(f,\cdot)$.

 (2)  $H_{n,m}(p,\cdot)=H_{n,m}(f,\cdot)$ if and only if $p=f$ or
 $$ \frac{\widehat{f_T^n}(\cdot)-\widehat{G^n}(\cdot)}{\widehat{f_T^m}(\cdot)-\widehat{G^m}(\cdot)}
 =H_{n,m}(f,\cdot).$$
\end{proposition}

\begin{remark}\label{rem:final}
Let $\mathcal H$ be the set of (two-segment) concatenated functions $$p(t)=F(t)\textbf{1}_{[0,T)}(t)+G(t-T)\textbf{1}_{[T,\infty)}(t)$$  with
$F,G\in\mathcal D$, $T>0$  and $F\not=0$.
Given any nonzero function $f\in \mathcal D$, let $$\mathcal H_f=\{p\in\mathcal H: H_{n,m}(p,\cdot)=H_{n,m}(f,\cdot)\}.$$
We would like to know what the set $\mathcal H_f$ is. To this purpose, we introduce
$$\mathcal D_f=\left\{p(t)=f(t)\textbf{1}_{[0,T)}(t)+G(t-T)\textbf{1}_{[T,\infty)}(t): p\in \mathcal H, p=f \text{\,or\,} \frac{\widehat{f_T^n}(\cdot)-\widehat{G^n}(\cdot)}{\widehat{f_T^m}(\cdot)-\widehat{G^m}(\cdot)}
=H_{n,m}(f,\cdot)\right\}.$$
Then using Proposition \ref{prop4.2} and \cite[Lemma 2.1, Corollary 2.3]{konstantopoulos2021does} we have
$\mathcal H_f= \mathcal D_f$ when $n-m$ is odd, and $\mathcal H_f= \mathcal D_f\bigcup(-\mathcal D_f)$ when $n-m$ is even. Here $-\mathcal D_f=\{-p: p\in \mathcal D_f\}.$ Therefore it suffices to determine $\mathcal D_f$ in order to find $\mathcal H_f$.
\end{remark}

Next we study an example for which $p\neq q$ and $H_{n,m}(p,\cdot)=H_{n,m}(q,\cdot).$
\begin{example}
Let $f(t)=e^{-at}$  for some $a\neq 0$.
Take any  function $p\not=f$ from the set $\mathcal D_f$.
Then  the function
 $p(t)=f(t)\textbf{1}_{[0,T)}(t)+G(t-T)\textbf{1}_{[T,\infty)}(t)$   satisfies
$$\frac{\widehat{f_T^n}(\lambda)-\widehat{G^n}(\lambda)}{\widehat{f_T^m}(\lambda)-\widehat{G^m}(\lambda)}=H_{n,m}(f,\lambda)=\frac {\lambda+ma}{\lambda+na}.$$
The above entails that $\lambda \widehat{G^n}(\lambda)-\lambda\widehat{G^m}(\lambda)+na\widehat{G^n}(\lambda)-ma\widehat{G^m}(\lambda)=e^{-naT}-e^{-maT}$, which is  equivalent to
  \begin{equation}\label{eqap1}
   n\widehat{G^{n-1}G'}(\lambda)- m\widehat{G^{m-1}G'}(\lambda)+na\widehat{G^n}(\lambda)-ma\widehat{G^m}(\lambda)=e^{-naT}-e^{-maT}+G^m(0)-G^n(0),
   \end{equation}
 by noting that $\lambda\widehat{G^n}(\lambda)=n\widehat{G^{n-1}G'}(\lambda)+G^n(0)$ and $\lambda\widehat{G^m}(\lambda)=m\widehat{G^{m-1}G'}(\lambda)+G^m(0)$.
 The right-hand side of  (\ref{eqap1}) is constant,
 the left side tends to zero as  $\lambda\to \infty$.
 Therefore both sides of (\ref{eqap1}) are constant zero, which implies
   $$ (nG^{n-1}(t)-mG^{m-1}(t))(G'(t)+aG(t))=e^{-naT}-e^{-maT}+G^m(0)-G^n(0)=0$$
   for all $t\ge 0$. Thus
   $G(t)=ce^{-at}$  for $t\ge 0$ with $c$ satisfying  $c^n-c^m=e^{-naT}-e^{-maT}$,
   or $G(t)=b$ for  $t\ge 0$ with $b$ satisfying  $b^n-b^m=e^{-naT}-e^{-maT}$ and $nb^{n-m}=m$.
   We are now ready to state $\mathcal D_f$ which is related to $\mathcal H_f$ as described in Remark \ref{rem:final}.
   Setting $$\mathcal G_a=\Big\{p(t)=e^{-at}\left(\textbf{1}_{[0,T)}(t)+ce^{aT}\textbf{1}_{[T,\infty)}(t)\right): T>0,c^n-c^m=e^{-naT}-e^{-maT}\Big\},$$
 and using Lemma \ref{lemma2.1}, we get that
  $$\mathcal D_f=\left\{
                   \begin{array}{ll}
                    % \mathcal G_a=\{f\} & \hbox {  $a<0$, $n$ is odd  and $m$ is even;}\\
                     \mathcal G_a, & \hbox{ $a<0$,  $n-m$ is odd;} \\
                     \mathcal G_a\cup\{p(t)=e^{-at}\textbf{1}_{[0,T_1)}(t)-b_0\textbf{1}_{[T_1,\infty)}(t)\}, & \hbox{ $a<0$,  $n$ and $m$ are odd;} \\
                     \mathcal G_a\cup \{p(t)=e^{-at}\textbf{1}_{[0,T_2)}(t)+b_0\textbf{1}_{[T_2,\infty)}(t)\}, & \hbox{ $a>0$;}
                   \end{array}
                 \right.
  $$
where $b_0=x_0$, $T_2=\frac {1}{(n-m)a}\ln \frac nm$ and $T_1>0$ satisfying $b^m_0-b^n_0=e^{-naT_1}-e^{-maT_1}$.
When $n$  is odd and $m$ is even, $\mathcal G_a(=\mathcal H_f)=\{f\}$ due to Remark \ref{rem:final} and Lemma \ref{lemma2.1}. That means if $H_{n,m}(p,\cdot)=H_{n,m}(q,\cdot),$ then $p=q=f.$
This  result is consistent with  Theorem \ref{theorem1.1}-(3).

This toy example shows that even in this simple situation, $p$ may not be unique and can take quite complicated forms.

%  and $p$ is positive near zero.
%If $a<0$, $n$ is odd and $m$ is even, then
%$p=f$.  In other cases,  $p=f$, or
 %there is $T>0$ and $c\not=e^{-aT}$ satisfying $c^n-c^m=e^{-naT}-e^{-maT}$ such that
%\begin{equation*}
%\begin{split}
% p(t)=e^{-at}\left(1_{[0,T)}(t)+ce^{aT}1_{[T,\infty)}(t)\right),
%\end{split}
%\end{equation*}
% or there are the following additional situations:
% (i)    $a<0$,  $n$ and $m$ are odds,
 %$$p(t)=e^{-at}1_{[0,T_1)}(t)-b_01_{[T_1,\infty)}(t),$$
 %where $b_0=(\frac mn)^{1/(n-m)}$ and $T_1>0$ satisfying $b^m_0-b^n_0=e^{-naT_1}-e^{-maT_1}$;
%(1) if $a<0$ and $n-m$ is odd,
%\begin{equation*}
%\begin{split}
 %p(t)=e^{-at}(1_{[0,T)}(t)+ce^{aT}1_{[T,\infty)}(t)),
%\end{split}
%\end{equation*}
%where $T>0,c^n-c^m=e^{-naT}-e^{-maT}$;
%(2) if $a<0$,  $n$ and $m$ are odds,
%\begin{equation*}
%\begin{split}
% p(t)=e^{-at}(1_{[0,T)}(t)+ce^{aT}1_{[T,\infty)}(t))\quad\text{or}\quad p(t)=e^{-at}1_{[0,T_1)}(t)-b_01_{[T_1,\infty)}(t),
%\end{split}
%\end{equation*}
%where $T>0,c^n-c^m=e^{-naT}-e^{-maT}$, $b_0=(\frac mn)^{1/(n-m)}$ and $T_1>0$ satisfying $b^m_0-b^n_0=e^{-naT_1}-e^{-maT_1}$;
%(ii)  $a>0$,
%\begin{equation*}
%\begin{split}
% p(t)=e^{-at}1_{[0,T_2)}(t)+b_01_{[T_2,\infty)}(t),
%\end{split}
%\end{equation*}
%where  $T_2=\frac {1}{(n-m)a}\ln \frac nm$.
%}
\end{example}

\bibliographystyle{abbrv}
	\bibliography{Ref}

\vspace*{1cm}

\noindent
\begin{minipage}{\textwidth}
\begin{minipage}{9cm}
Dongdong Hu\\
Department of Financial and Actuarial Mathematics\\
Xi'an Jiaotong - Liverpool University\\
111 Ren'ai Road, Suzhou, P.\ R.\ China 215123\\
\href{mailto:dongdong.hu20@student.xjtlu.edu.cn}{dongdong.hu20@student.xjtlu.edu.cn}
\end{minipage}
\begin{minipage}{7cm}
Linglong Yuan\\
Department of Mathematical Sciences\\
University of Liverpool\\
Liverpool  L69 7ZL, UK\\
\href{mailto:linglong.yuan@liverpool.ac.uk}{linglong.yuan@liverpool.ac.uk}
\end{minipage}
\begin{minipage}{9cm}
Minzhi Zhao\\
School of Mathematical Sciences\\
Zhejiang University\\
866 Yuhangtang Rd, Hangzhou, P. R. China 310058\\
\href{zhaomz@zju.edu.cn}{zhaomz@zju.edu.cn}
\end{minipage}

\end{minipage}

\end{document}